\documentclass{amsart}
\usepackage{amsfonts}
\usepackage{amsmath,amssymb}
\usepackage{amsthm}
\usepackage{amscd}
\usepackage{graphics}
\usepackage{graphicx}

\theoremstyle{remark}{
\newtheorem{Def}{{\rm Definition}}
\newtheorem{Ex}{{\rm Example}}
\newtheorem{Rem}{{\rm Remark}}

}
\theoremstyle{plain}
{
\newtheorem{Cor}{Corollary}
\newtheorem{Prop}{Proposition}
\newtheorem{Thm}{Theorem}

}

\begin{document}
\title[Refined algebraic domains with characteristic finite sets]{Refined algebraic domains with finite sets in the boundaries}
\author{Naoki kitazawa}

\address{Institute of Mathematics for Industry, Kyushu University, 744 Motooka, Nishi-ku Fukuoka 819-0395, Japan\\
 TEL (Office): +81-92-802-4402 \\
 FAX (Office): +81-92-802-4405 \\
}
\email{n-kitazawa@imi.kyushu-u.ac.jp, naokikitazawa.formath@gmail.com}
\urladdr{https://naokikitazawa.github.io/NaokiKitazawa.html}
\maketitle
\begin{abstract}
{\it Refined algebraic domains} are regions in the plane surrounded by finitely many non-singular real algebraic curves which may intersect with normal crossing. We are interested in shapes of such regions with surrounding real algebraic curves. {\it Poincar\'e-Reeb Graphs} of them are graphs the regions naturally collapse to respecting the projection to a straight line. Such graphs were first formulated by Sorea, for example, around 2020, and regions surrounded by mutually disjoint non-singular real algebraic curves were mainly considered. The author has generalized the studies to several general situations. 

We find classes of such objects defined inductively by adding curves. We respect characteristic finite sets in the curves. We consider regions surrounded by the curves and of a new type. We investigate geometric properties and combinatorial ones of them and discuss important examples. We also previously studied explicit classes defined inductively in this way and review them.

\end{abstract}
\section{Introduction.}
\label{sec:1}
The (interior of the) unit disk is one of simplest bounded regions in the plane. Regions in the plane surrounded by (so-called) {\it non-singular} real algebraic curves are generalizations, and fundamental spaces and objects in algebraic geometry, mainly. \cite{bodinpopescupampusorea, sorea1, sorea2} present elementary, fundamental, natural, and surprisingly new understanding of their shapes, especially, convexity. These regions are called {\it algebraic domains}, for example. Graphs they collapse to respecting the projection to the straight line $\{(t,0) \mid t \in \mathbb{R}\}$ are introduced and important there: in our paper let ${\mathbb{R}}^n$ ($\mathbb{R}:={\mathbb{R}}^1$) denote the $n$-dimensional Euclidean space, being also a Riemannian manifold equipped with the standard Euclidean metric, and $||x|| \geq 0$ the distance between $x \in {\mathbb{R}}^n$ and the origin $0 \in {\mathbb{R}}^n$ under this metric. 
For example, so-called generic embedding of graphs into ${\mathbb{R}}^2$ with respect to the projection to the line $\{(t,0) \mid t \in \mathbb{R}\}$ have been shown to be realized as Poincar\'e-Reeb graphs of some real algebraic domains, by several technique of real algebraic approximation. Note that including this exposition, the content of present section respects exposition of \cite{kitazawa8} mainly and related studies \cite{kitazawa5, kitazawa6, kitazawa7} of the author.

\subsection{Our fundamental terminologies, notions and notation.}
We use ${\pi}_{m,n}:{\mathbb{R}}^m \rightarrow {\mathbb{R}}^n$ with $m>n \geq 1$ for the so-called canonical projection ${\pi}_{m,n}(x)=x_1$ with $x=(x_1,x_2) \in {\mathbb{R}}^n \times {\mathbb{R}}^{m-n}={\mathbb{R}}^m$. We represent the $k$-dimensional unit disk by $D^k:=\{x \in {\mathbb{R}}^k \mid ||x|| \leq 1\}$ and the $k$-dimensional unit sphere by $S^k:=\{x \in {\mathbb{R}}^{k+1} \mid ||x||=1\}$. 

Let $X$ be a topological space and $Y \subset X$ a subspace. We use $\overline{Y}$ for the closure and $Y^{\circ}$ for the interior in $X$ where $X$ is a Euclidean space considered in the discussion unless otherwise stated. For a topological space $X$ having the structure of a so-called cell complex, we can define the dimension $\dim X$ uniquely as the dimension of the cell of the maximal dimension and this only depends on the topology of $X$. Topological manifolds, polyhedra, and graphs, regarded as $1$-dimensional CW complexes, are examples for such spaces. 
For a smooth manifold $X$ and a point $x \in X$, we use $T_xX$ for the tangent vector space of $X$ at $x$. For smooth manifolds $X$ and $Y$ and a smooth map $c:X \rightarrow Y$, a point $x \in X$ is called a {\it singular point} of $c$ if the rank of the differential ${dc}_x:T_xX \rightarrow T_{c(x)}Y$ is smaller than the minimum between the dimensions $\dim X$ and $\dim Y$: remember that the differential ${dc}_x$ is linear.
A union $S$ of connected components of the zero set of a real polynomial map is {\it non-singular} if the polynomial map has no singular point in the set $S$: remember the implicit function theorem.  

A {\it graph} is a CW complex of dimension $1$ with $1$-cells ({\it edges}) and $0$-cells ({\it vertices}). The set of all edges (vertices) of the graph is the {\it edge set} ({\it vertex set}) of it. Two graphs $G_1$ and $G_2$ are {\it isomorphic} if there exists a (piecewise smooth) homeomorphism $\phi:G_1 \rightarrow G_2$ mapping the vertex set of $G_1$ onto that of $G_2$ and this is called an {\it isomorphism} of the graphs. A {\it digraph} is a graph all of whose edges are oriented. Two digraphs are {\it isomorphic} if there exists an isomorphism of graphs preserving the orientations and this is called an {\it isomorphism} of the digraphs.
\subsection{Refined algebraic domains.}
\label{subsec:1.2}
\begin{Def}
\label{def:1}
In the present paper, a pair of a family $\mathcal{S}=\{S_j \subset {\mathbb{R}}^2\}$ each $S_j$ of which is a connected component of a real polynomial and non-singular and a region $D_{\mathcal{S}} \subset {\mathbb{R}}^2$ satisfying the following conditions is called a {\it refined algebraic domain}. 
\begin{enumerate}
\item The region $D_{\mathcal{S}}$ is a bounded connected component of ${\mathbb{R}}^2-{\bigcup}_{S_j \in \mathcal{S}} S_j$ such that the intersection $\overline{D_{\mathcal{S}}} \bigcap S_j$ is non-empty for any curve $S_j \in \mathcal{S}$.
\item At each point $p_{j_1,j_2} \in \overline{D_{\mathcal{S}}}$, at most two distinct curves $S_{j_1}, S_{j_2} \in \mathcal{S}$ intersect enjoying the following properties: for each $p_{j_1,j_2}$ of the intersection of distinct two curves from $\mathcal{S}$, the sum of the tangent vector spaces of them at $p_{j_1,j_2}$ and the tangent vector space of ${\mathbb{R}}^2$ at $p_{j_1,j_2}$ coincide. 
\end{enumerate}
\end{Def}

Related to this, \cite{kohnpieneranestadrydellshapirosinnsoreatelen} studies regions regarded as refined algebraic domains in our paper, explicitly and systematically, for example. We discuss the restriction of ${\pi}_{2,1,i}$ to $\overline{D_{\mathcal{S}}}$. We consider the set $F_{D_{\mathcal{S}},i}$ of all points in the following. This set is finite since we only consider real algebraic objects.
\begin{itemize}
\item 
Points in $\overline{D_{\mathcal{S}}}$ which are also in exactly two distinct curves $S_{j_1}$ and $S_{j_2}$.
\item By removing the set of all points before from the set $\overline{D_{\mathcal{S}}}-D_{\mathcal{S}}$ of dimension $1$, we have a smooth manifold of dimension $1$ with no boundary. Points which are singular points of the restriction of ${\pi}_{2,1,i}$ to the obtained smooth curve in $\overline{D_{\mathcal{S}}}-D_{\mathcal{S}}$ and which are in the curve $S_j$ not being a disjoint union of connected components of the zero set of a real polynomial of degree $1$. 
\end{itemize} 
We can define the equivalence relation ${\sim}_{D_{\mathcal{S}},i}$ on $\overline{D_{\mathcal{S}}}$ as follows. Two points are equivalent if and only if they are in a same component of the preimage of a same point for the restriction of ${\pi}_{2,i}$ to $\overline{D_{\mathcal{S}}}$. Let $q_{D_{\mathcal{S}},i}$ denote the quotient map and we can define the function $V_{D_{\mathcal{S}},i}$ with the relation ${\pi}_{2,i}=V_{D_{\mathcal{S}},i} \circ q_{D_{\mathcal{S}},i}$ uniquely. The quotient space $W_{D_{\mathcal{S}},i}:=\overline{D_{\mathcal{S}}}/{\sim}_{D_{\mathcal{S}},i}$ is a digraph. We can check this from general theory \cite{saeki1, saeki2} or see \cite{kitazawa5} for example: we do not need to understand this theory.
\begin{enumerate}
\item The vertex set is the set of all points $v$ such that the preimage ${q_{D_{\mathcal{S}},i}}^{-1}(v)$ contains at least one point of the finite set $F_{D_{\mathcal{S}},i}$.
\item The edge connecting $v_1$ and $v_2$ are oriented as one departing from $v_1$ and entering $v_2$ according to $V_{D_{\mathcal{S}},i}(v_1)<V_{D_{\mathcal{S}},i}(v_2)$.
\end{enumerate}
\begin{Def}
\label{def:2}
We call the (di)graph $(W_{D_{\mathcal{S}},i},V_{D_{\mathcal{S}},i})$ a {\it Poincar\'e-Reeb }({\it di}){\it graph of $D_{\mathcal{S}}$}. We omit the function $V_{D_{\mathcal{S}},i}$ where we can guess easily.
\end{Def}
More generally, for a graph $G$ and a map $V_{G}$ on its vertex set onto a partially ordered set $P$ with natural conditions, we can orient the graph according to the values. More precisely, each edge $e$ of the graph connects two distinct vertices $v_{e,1}$ and $v_{e,2}$ and it is oriented according to the rule that the edge $e$ departs from $v_{e,1}$ and enters $v_{e,2}$ if and only if $V_{G}(v_{e,1})<V_{G}(v_{e,2})$: let "$<$" denote the order on $P$. 
A pair of this graph $G$ and a map $V_{G}$ is said to be {\it a V-digraph}. For V-graphs, {\it isomorphisms} between two V-digraphs and the relation that two V-digraphs are {\it isomorphic} are canonically defined, for example.

For original studies on Poincar'e-Reeb graphs, consult \cite{bodinpopescupampusorea, sorea1, sorea2}.
\subsection{Our main work.}
In our paper, we consider a characteristic finite set $A_{D_{\mathcal{S}}}$ in $\overline{D_{\mathcal{S}}}-D_{\mathcal{S}}$: $A_{D_{\mathcal{S}}}:=F_{D_{\mathcal{S}},1} \bigcup F_{D_{\mathcal{S}},2}$ for example. This idea is based on arguments first presented in \cite{kitazawa8} where curves are circles of fixed radii. There points of the form $({x_j}_1+r_j \cos (\frac{\pi a}{4}),{x_j}_2+r_j \sin (\frac{\pi a}{4}))$ with $x_j:=({x_j}_1,{x_j}_2)$, $r_j>0$, and $a=0,1,2,3,4,5,6,7$, are also considered for $A_{D_{\mathcal{S}}} \supset F_{D_{\mathcal{S}},1} \bigcup F_{D_{\mathcal{S}},2}$. As done there with a related pioneering study \cite{kitazawa5} by the author and related studies \cite{kitazawa6, kitazawa8}, following \cite{kitazawa5}, we consider adding curves to existing pairs of regions and curves surrounding the regions, according to certain rules. We investigate suitable classes of such rules. We define classes and investigate geometric properties and combinatorial ones.
We also discuss examples of these classes of rules for the changes of refined algebraic domains with the finite sets, for example: refined algebraic domains with the finite sets are named {\it refined algebraic domains with poles} (Definition \ref{def:3}). 

The next section is devoted to definitions of our new classes (Definitions \ref{def:3} and \ref{def:4}) and our fundamental result (Corollary \ref{cor:1} and Theorem \ref{thm:1}). The third section concentrates on cases where the curves are circles or more generally, the boundaries of ellipsoids, discuss our classes they belong to and have our new result (Theorems \ref{thm:2} and \ref{thm:3}). 
Adding "sufficiently small" closed disks, whose boundaries are circles, in certain rules, and changes of the regions and the Poincar\'e-Reeb V-graphs by the addition, has been studied explicitly and systematically, in \cite{kitazawa5, kitazawa6, kitazawa8}. Such existing rules are reviewed in our stream (Remark \ref{rem:1}).
The fourth section discusses our classes for a case from \cite[Theorem 1]{kitazawa7} (Theorem \ref{thm:4}) and a new revised case (Theorem \ref{thm:5}). Note that our previous studies \cite{kitazawa5, kitazawa6, kitazawa7, kitazawa8} and the present study are originally, motivated by singularity theory of differentiable functions and maps (\cite{kitazawa1, kitazawa2, kitazawa3, kitazawa4}). In short, the closures of the regions are the images of natural real algebraic maps locally like so-called moment maps and generalizing the canonical projections ${\pi}_{m+1,n,i} {\mid}_{S^m}:S^m \rightarrow {\mathbb{R}}^n$ of the unit spheres $S^m \subset {\mathbb{R}}^{m+1}$. We are interested in explicit construction of nice real algebraic maps and understanding their "shapes" and "structures", which is fundamental, natural, and still difficult. We explain our related previous result (Theorem \ref{thm:6}) and an explicit fact related to Theorems \ref{thm:4} and \ref{thm:5} from the viewpoint of our present study.


\section{Refined algebraic domains with poles and classes of them.}
\label{sec:2}
\begin{Def}
\label{def:3}
In Definition \ref{def:1}, let $A_{D_{\mathcal{S}}} \subset \overline{D_{\mathcal{S}}}-D_{\mathcal{S}}$ be a finite set such that for the finite sets $F_{D_{\mathcal{S}},i}$ before, $F_{D_{\mathcal{S}},1} \bigcup F_{D_{\mathcal{S}},2} \subset A_{D_{\mathcal{S}}}$ is satisfied. We call the triplet $(\mathcal{S},D_{\mathcal{S}},A_{D_{\mathcal{S}}})$ a {\it refined algebraic domain with poles}.
\end{Def}
\begin{Def}
\label{def:4}
In Definition \ref{def:3}, for a connected component $S_{j^{\prime}}$ of the zero set of a real polynomial $f_{j^{\prime}}$ and a connected component $D_{j^{\prime}}$ of the complementary set ${\mathbb{R}}^2-S_{j^{\prime}}$, let $A_{j^{\prime}} \subset S_{j^{\prime}}$ be a finite set such that $(\{S_{j^{\prime}}\},D_{j^{\prime}},A_{j^{\prime}})$ is a refined algebraic domain with poles, and suppose that $({\mathcal{S}}^{\prime}:=\mathcal{S} \sqcup \{S_{j^{\prime}}\}, D_{{\mathcal{S}}^{\prime}}:=D_{\mathcal{S}} \bigcap ({\mathbb{R}}^2-\overline{D_{j^{\prime}}}), A_{D_{{\mathcal{S}}^{\prime}}}:=(A_{D_{\mathcal{S}}} \bigcup A_{j^{\prime}}) \bigcap \overline{D_{{\mathcal{S}}^{\prime}}})$ is a refined algebraic domain with poles.
\begin{enumerate}
\item The set $D_{j^{\prime}}$ is said to be {\it $(\mathcal{S},D_{\mathcal{S}})$-connected} if the intersection $\overline{D_{j^{\prime}}} \bigcap \overline{D_{\mathcal{S}}}$ is connected.    
\item The pointed set $(D_{j^{\prime}},x_{j^{\prime}})$ with $x_{j^{\prime}}=({x_{j^{\prime}}}_1,{x_{j^{\prime}}}_2) \in D_{j^{\prime}} \bigcap  (\overline{D_{\mathcal{S}}}-D_{\mathcal{S}})$ is said to be {\it small with respect to $(\mathcal{S},D_{\mathcal{S}},A_{D_{\mathcal{S}}})$} or {\it $(\mathcal{S},D_{\mathcal{S}},A_{D_{\mathcal{S}}})$-S} if each set ${\pi}_{2,1,i}(A_{D_{\mathcal{S}}}) \bigcap {\pi}_{2,1,i}(\overline{D_{j^{\prime}}} \bigcap \overline{D_{\mathcal{S}}})-\{{x_{j^{\prime}}}_i\}$ is empty for $i=1,2$. If in addition, $D_{j^{\prime}}$ is $(\mathcal{S},D_{\mathcal{S}})$-connected, then it is said to be {\it properly small with respect to $(\mathcal{S},D_{\mathcal{S}},A_{D_{\mathcal{S}}})$} or {\it $(\mathcal{S},D_{\mathcal{S}},A_{D_{\mathcal{S}}})$-PS}.
\item The pointed set $(D_{j^{\prime}},x_{j^{\prime}})$ with $x_{j^{\prime}}=({x_{j^{\prime}}}_1,{x_{j^{\prime}}}_2) \in D_{j^{\prime}} \bigcap (\overline{D_{\mathcal{S}}}-D_{\mathcal{S}})$ is said to be {\it locally small with respect to $(\mathcal{S},D_{\mathcal{S}},A_{D_{\mathcal{S}}})$} or {\it $(\mathcal{S},D_{\mathcal{S}},A_{D_{\mathcal{S}}})$-LS} if each set $A_{D_{\mathcal{S}}} \bigcap \overline{D_{j^{\prime}}} \bigcap \overline{D_{\mathcal{S}}}-\{x_{j^{\prime}}\}$ is empty or a finite set of points of the form $({x_{j^{\prime}}}_1,{{x_{j^{\prime}}}_2}^{\prime})$ or $({{x_{j^{\prime}}}_1}^{\prime},{x_{j^{\prime}}}_2)$ with arbitrary numbers ${{x_{j^{\prime}}}_i}^{\prime} \in \mathbb{R}$. If in addition, $D_{j^{\prime}}$ is $(\mathcal{S},D_{\mathcal{S}})$-connected, then it is said to be {\it properly and locally small with respect to $(\mathcal{S},D_{\mathcal{S}},A_{D_{\mathcal{S}}})$} or {\it $(\mathcal{S},D_{\mathcal{S}},A_{D_{\mathcal{S}}})$-PLS}.
\end{enumerate}
\end{Def}
From the definition, we have the following immediately.
\begin{Cor}
	\label{cor:1}
In Definition \ref{def:4}, the pointed set $(D_{j^{\prime}},x_{j^{\prime}})$ with $x_{j^{\prime}}=({x_{j^{\prime}}}_1,{x_{j^{\prime}}}_2) \in D_{j^{\prime}} \bigcap {\bigcup}_{S_j \in \mathcal{S}} S_j$ which is $(\mathcal{S},D_{\mathcal{S}},A_{D_{\mathcal{S}}})$-LS {\rm (}$(\mathcal{S},D_{\mathcal{S}},A_{D_{\mathcal{S}}})$-S{\rm )} is also $(\mathcal{S},D_{\mathcal{S}},A_{D_{\mathcal{S}}})$-PLS {\rm (}resp. $(\mathcal{S},D_{\mathcal{S}},A_{D_{\mathcal{S}}})$-PS{\rm )}.

\end{Cor}

Most of our new result consists of checking examples of refined algebraic domains with poles of some of the defined classes.

We assume knowledge on fundamental arguments from singularity theory and real algebraic geometry.
For singularity theory of differentiable maps, consult \cite{golubitskyguillemin} for example. \cite{bochnakcosteroy, elredge, kollar, lellis} are for real polynomials, especially, their approximation theory, and general theory of real algebraic geometry. 

As a kind of elementary mathematics, we also assume knowledge on geometry of the Euclidean plane ${\mathbb{R}}^2$, for example. For example, we need the notion of {\it parallel} subsets there. This also comes from the structure of the Riemannian manifold. We also need the notion of similarity.

An {\it ellipsoid of ${\mathbb{R}}^2$ centered at a point $x_0=(x_{0,1},x_{0,2}) \in {\mathbb{R}}^2$} means a set of the form $\{x=(x_1,x_2) \in {\mathbb{R}}^2 \mid a_1{(x_1-x_{0,1})}^2+a_2{(x_2-x_{0,2})}^2 \leq r\}$ with $a_1,a_2,r>0$ or a subset obtained by rotating the set around the point $x_0=(x_{0,1},x_{0,2}) \in {\mathbb{R}}^2$. We call an ellipsoid centered at $x_0$ of the form $\{x=(x_1,x_2) \in {\mathbb{R}}^2 \mid a_1{(x_1-x_{0,1})}^2+a_2{(x_2-x_{0,2})}^2 \leq r\}$ with $a_1,a_2,r>0$ an {\it ellipsoid of the standard form}. We use these terminologies on ellipsoids for the interiors of these closed sets in ${\mathbb{R}}^2$. 
As a specific case of the boundaries of ellipsoids (centered at some points), a {\it circle of ${\mathbb{R}}^2$ centered at a point $x_0=(x_{0,1},x_{0,2}) \in {\mathbb{R}}^2$} means a set of the form $\{x=(x_1,x_2) \in {\mathbb{R}}^2 \mid {(x_1-x_{0,1})}^2+{(x_2-x_{0,2})}^2 \leq r\}$ with $r>0$.
\begin{Thm}
\label{thm:1}
In Definition \ref{def:4}, for an arbitrary point $x_{j^{\prime}}:=({x_{j^{\prime}}}_1,{x_{j^{\prime}}}_2) \in D_{j^{\prime}} \bigcap (\overline{D_{\mathcal{S}}}-D_{\mathcal{S}}) \subset {\mathbb{R}}^2$ and any point $x_{j^{\prime},1}:=(x_{j^{\prime},1,1},x_{j^{\prime},1,2}) \in {\mathbb{R}}^2-(F_{x_{j^{\prime}}} \bigcup \{x_{j^{\prime}}\})$ with a suitably chosen union $F_{x_{j^{\prime}}}$ of finitely many straight lines, we can have a pointed set $(D_{j^{\prime}},x_{j^{\prime}})$ with $x_{j^{\prime}} \in D_{j^{\prime}} \bigcap (\overline{D_{\mathcal{S}}}-D_{\mathcal{S}})$ enjoying the following.
\begin{enumerate}
\item The pointed set $(D_{j^{\prime}},x_{j^{\prime}})$ is $(\mathcal{S},D_{\mathcal{S}},A_{D_{\mathcal{S}}})$-LS.
\item The set $D_{j^{\prime}}$ is an ellipsoid centered at the point $x_{j^{\prime},0}:=(\frac{{x_{j^{\prime}}}_1+x_{j^{\prime},1,1}}{2},\frac{{x_{j^{\prime}}}_2+x_{j^{\prime},1,2}}{2}) \in {\mathbb{R}}^2$ and containing the straight segment connecting $x_{j^{\prime}}$ and $x_{j^{\prime},1}$.
\end{enumerate}
In addition, the pointed set $(D_{j^{\prime}},x_{j^{\prime}})$ with $x_{j^{\prime}}=({x_{j^{\prime}}}_1,{x_{j^{\prime}}}_2) \in D_{j^{\prime}} \bigcap (\overline{D_{\mathcal{S}}}-D_{\mathcal{S}})$ can be chosen to be $(\mathcal{S},D_{\mathcal{S}},A_{D_{\mathcal{S}}})$-S if the following are satisfied.
\begin{enumerate}
	\setcounter{enumi}{2}
\item The straight segment connecting $x_{j^{\prime}}$ and $x_{j^{\prime},1}$ 
 is parallel to a segment in $\mathbb{R} \times \{0\} \subset {\mathbb{R}}^2$ or  $\{0\} \times \mathbb{R} \subset {\mathbb{R}}^2$.
\item The straight segment connecting $x_{j^{\prime}}$ and $x_{j^{\prime},1}$ contains at most two points from $\overline{D_{\mathcal{S}}}${\rm :} the two points must be from $\{x_{j^{\prime}}, x_{j^{\prime},1}\} \bigcap (\overline{D_{\mathcal{S}}}-D_{\mathcal{S}})$.
\end{enumerate}
\end{Thm} 
\begin{proof}
By the finiteness, for a suitably chosen union $F_{x_{j^{\prime}}}$ of finitely many straight lines and any point $x_{j^{\prime},1}:=(x_{j^{\prime},1,1},x_{j^{\prime},1,2}) \in {\mathbb{R}}^2-(F_{x_{j^{\prime}}} \bigcup \{x_{j^{\prime}}\})$, we can see that our ellipsoid centered at the point $x_{j^{\prime},0}:=(\frac{{x_{j^{\prime}}}_1+x_{j^{\prime},1,1}}{2},\frac{{x_{j^{\prime}}}_2+x_{j^{\prime},1,2}}{2}) \in {\mathbb{R}}^2$ can be chosen as one containing at most one point from $A_{D_{\mathcal{S}}}$, which is $x_{j^{\prime}}:=({x_{j^{\prime}}}_1,{x_{j^{\prime}}}_2) \in {\mathbb{R}}^2$. We have the additional result immediately from our definition.
\end{proof}
\section{The case where the families $\mathcal{S}$ and ${\mathcal{S}}^{\prime}$ of Definition \ref{def:4} consist of circles.}

We first show an important example.
\begin{Ex}
\label{ex:1}
FIGURE \ref{fig:1} shows a region surrounded by two circles centered at a same point $x_0$, whose radii are sufficiently close, a disk $D_{j^{\prime}}$, colored in blue, and the black colored boundary of the closed disk $\overline{D_{j^{\prime}}}$.
This is $(\mathcal{S},D_{\mathcal{S}},A_{D_{\mathcal{S}}}:=F_{D_{\mathcal{S}},1} \bigcup F_{D_{\mathcal{S}},2})$-PS.

\begin{figure}
	\includegraphics[width=80mm,height=80mm]{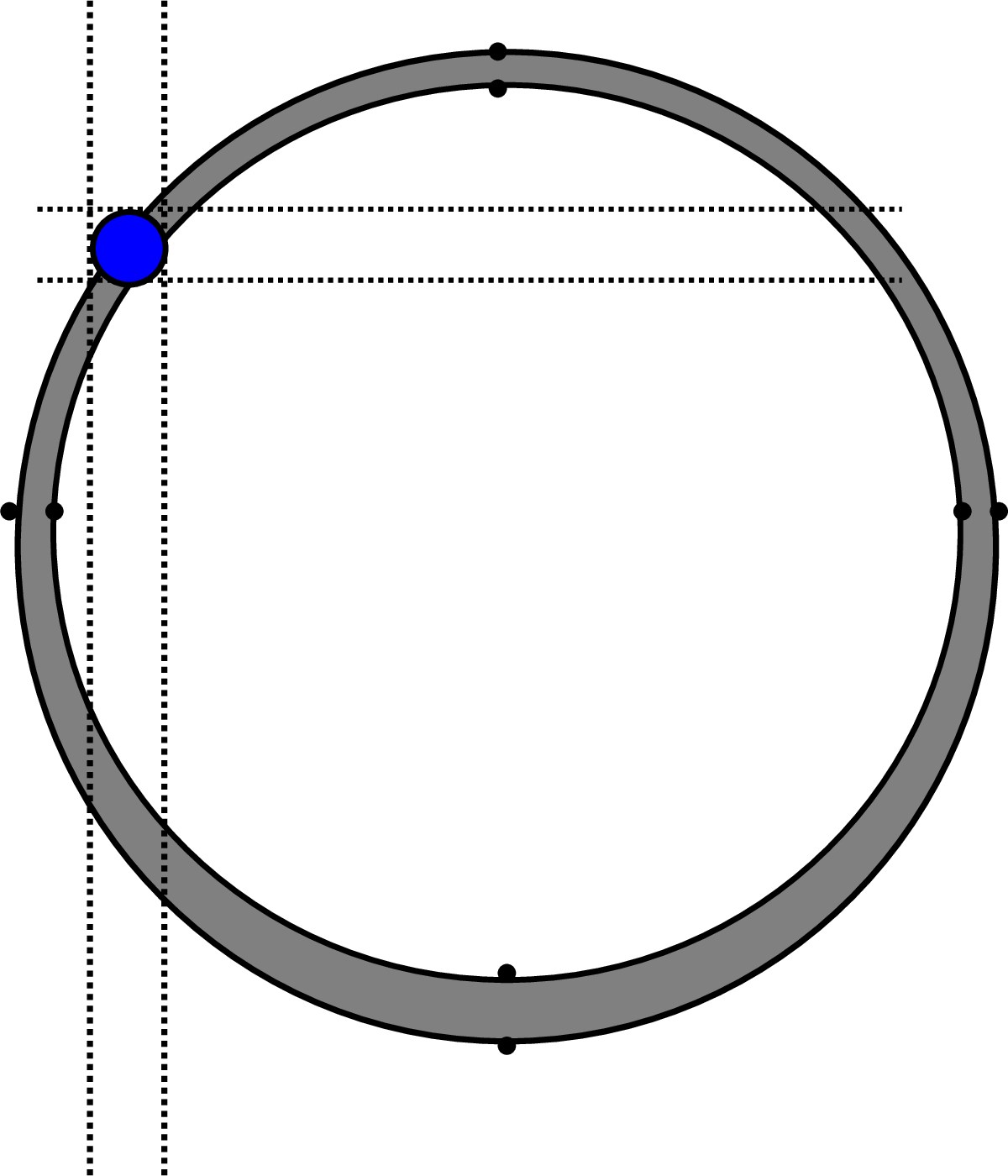}
	\caption{A region surrounded by two concentric circles whose radii are sufficiently close, a closed disk $D_{j^{\prime}}$, colored in blue, and the black colored boundary of the closed disk $\overline{D_{j^{\prime}}}$. Black dots are for points of $F_{D_{\mathcal{S}},1} \bigcup F_{D_{\mathcal{S}},2}$.}
	\label{fig:1}
\end{figure}
Note that this cannot be regarded as one from any family ${\mathcal{S}}^{\prime}$ from existing studies \cite{kitazawa5, kitazawa6, kitazawa8}. In these studies, we have considered the case a circle of a sufficiently small radius is put. We have investigated local changes of the Poincar\'e-Reeb V-digraphs for ${\pi}_{2,1,i}$. This is "small" under our new definition.
\end{Ex}

\begin{Thm}
\label{thm:2}
Let the families $\mathcal{S}$ and ${\mathcal{S}}^{\prime}$ in Definition \ref{def:4} consist of circles. Let $A_{D_{\mathcal{S}}}:=F_{D_{\mathcal{S}},1} \bigcup F_{D_{\mathcal{S}},2}$.
If the pointed set $(D_{j^{\prime}},x_{j^{\prime}})$ with $x_{j^{\prime}}=({x_{j^{\prime}}}_1,{x_{j^{\prime}}}_2) \in D_{j^{\prime}} \bigcap (\overline{D_{\mathcal{S}}}-D_{\mathcal{S}})$ is $(\mathcal{S},D_{\mathcal{S}},A_{D_{\mathcal{S}}})$-LS, and the set $D_{j^{\prime}}$ is not $(\mathcal{S},D_{\mathcal{S}})$-connected, then the following are satisfied.
\begin{enumerate}
\item \label{thm:2.1}
 The intersection $\overline{D_{j^{\prime}}} \bigcap (\overline{D_{\mathcal{S}}}-D_{\mathcal{S}})$ contains no circle.
\item \label{thm:2.2} The intersection $\overline{D_{j^{\prime}}} \bigcap (\overline{D_{\mathcal{S}}}-D_{\mathcal{S}})$ contains at least $l=2$ points contained in exactly two circles from $\mathcal{S}$. Furthermore, the number $l=2$ is optimal.
\end{enumerate}
\end{Thm}
\begin{proof}
The first property (\ref{thm:2.1}) immediately follows from the assumption that the pointed set $(D_{j^{\prime}},x_{j^{\prime}})$ with $x_{j^{\prime}}=({x_{j^{\prime}}}_1,{x_{j^{\prime}}}_2) \in D_{j^{\prime}} \bigcap (\overline{D_{\mathcal{S}}}-D_{\mathcal{S}})$ is $(\mathcal{S},D_{\mathcal{S}},A_{D_{\mathcal{S}}})$-LS and our definition.

We check (\ref{thm:2.2}). Suppose that the set $D_{j^{\prime}}$ is not $(\mathcal{S},D_{\mathcal{S}})$-connected and that the intersection $\overline{D_{j^{\prime}}} \bigcap (\overline{D_{\mathcal{S}}}-D_{\mathcal{S}})$ contains at most $l=1$ point. Under this assumption, we have the following two connected components $C_{\mathcal{S},1}$ and $C_{\mathcal{S},2}$ of $\overline{D_{j^{\prime}}} \bigcap (\overline{D_{\mathcal{S}}}-D_{\mathcal{S}})$.
\begin{itemize}
\item A smooth connected curve $C_{\mathcal{S},1}$ which connects two points of $(\overline{D_{j^{\prime}}}-D_{j^{\prime}}) \bigcap (\overline{D_{\mathcal{S}}}-D_{\mathcal{S}})$ and whose interior is in $D_{j^{\prime}} \bigcap (\overline{D_{\mathcal{S}}}-D_{\mathcal{S}})$.
\item Another curve $C_{\mathcal{S},2}$ which is piecewise smooth and connected, and connects two points of $(\overline{D_{j^{\prime}}}-D_{j^{\prime}}) \bigcap (\overline{D_{\mathcal{S}}}-D_{\mathcal{S}})$, and whose interior is in $D_{j^{\prime}} \bigcap (\overline{D_{\mathcal{S}}}-D_{\mathcal{S}})$ and contains a point contained in exactly two circles from $\mathcal{S}$. 
\end{itemize}
We can also choose these two curves in such a way that the following hold.
\begin{itemize}
\item In addition, $C_{\mathcal{S},1}$ and $C_{\mathcal{S},2}$ are disjoint and divide the disk $\overline{D_{j^{\prime}}}$ into three connected closed subsets $D_{j} \subset \overline{D_{j^{\prime}}}$ ($j=0,1,2$) satisfying the following conditions.
\begin{itemize}
\item The intersection $D_1 \bigcap D_2$ is empty.
\item The intersections $D_0 \bigcap D_1$ and $D_0 \bigcap D_2$ are non-empty.
\end{itemize}
\item For each $j=1,2$, the intersection of the component $D_0$ and a small neighborhood $N(C_{\mathcal{S},j}) \subset \overline{D_{j^{\prime}}}$ of $C_{\mathcal{S},j}$ and the open set $D_{\mathcal{S}}$ is disjoint. 
\item For each $j=1,2$, the intersection of the component $D_j$, the small neighborhood $N(C_{\mathcal{S},j}) \subset \overline{D_{j^{\prime}}}$ of $C_{\mathcal{S},j}$, and the open set $D_{\mathcal{S}}$ is non-empty. 
\end{itemize}
The curve $C_{\mathcal{S},1}$ and a smooth curve ${C_{\mathcal{S},2}}^{\prime} \subset C_{\mathcal{S},2}$ are in unique circles $S_{\mathcal{S},1}$ and $S_{\mathcal{S},2}$ from $\mathcal{S}$, respectively. We can also see that for these smooth curves $C_{\mathcal{S},1}$ and ${C_{\mathcal{S},2}}^{\prime}$ the circles are distinct: if they are not distinct, then the circle of $\mathcal{S}$ and the circle $\overline{D_{j^{\prime}}}-D_{j^{\prime}}$ must intersect at a finite set containing at least $3$ points and this is a contradiction.
Suppose that $S_{\mathcal{S},1}$ and $S_{\mathcal{S},2}$ do not intersect in any non-empty finite set in ${\mathbb{R}}^2$. By considering the conditions above, the region $D_{\mathcal{S}}$ must not be connected.
By considering the conditions above and properties on the shape of $D_{\mathcal{S}}$, $S_{\mathcal{S},1}$ and $S_{\mathcal{S},2}$ cannot intersect in any non-empty finite set in ${\mathbb{R}}^2$. This is a kind of elementary arguments on the intersection of two circles.
For such an intersection of circles and the shapes of regions like $D_{\mathcal{S}}$, check also \cite[FIGURE 7]{kitazawa5}, for example: we do not assume non-trivial arguments from the preprint \cite{kitazawa5}. 

This is a contradiction. We have checked $l \geq 2$ must hold.

We see that the assumption on the number $l \geq 2$ is optimal.

We consider the family $\mathcal{S}$ of circles consisting of exactly three circles and satisfying the following.
\begin{itemize}
\item Two circles $S_1$ and $S_2$ from $\mathcal{S}$ are in the bounded connected component of the complementary set ${\mathbb{R}}^2-S_0$ of the remaining circle $S_0$ of $\mathcal{S}$.
\item The two circles $S_1$ and $S_2$ of $\mathcal{S}$ above are of a same radius and centered at two points in a line parallel to the horizontal line $\mathbb{R} \times \{0\}$ and they intersect at a two-point set $F_{\mathcal{S}}$. 
\item The region $D_{\mathcal{S}}$ is the unique connected component of the complementary set ${\mathbb{R}}^2-{\bigcup}_{S_j \in \mathcal{S}} S_j$ of ${\bigcup}_{S_j \in \mathcal{S}} S_j$.
\end{itemize}
We put a disk $D_{j^{\prime}}$ with a point $x_{j^{\prime}} \in D_{j^{\prime}}$ contained in the two circles $S_1$ and $S_2$ from $\mathcal{S}$. We put $D_{j^{\prime}}$ in such a way that the two-point set $F_{\mathcal{S}}$ before is contained there and that the circle $\overline{D_{j^{\prime}}}-D_{j^{\prime}}$ of the boundary of $\overline{D_{j^{\prime}}}$ is centered at a point in the unique straight segment connecting the two points of $F_{\mathcal{S}}$ and of a sufficiently small radius. We can have the set $D_{j^{\prime}} \bigcap (\overline{D_{\mathcal{S}}}-D_{\mathcal{S}})$ containing no point from $A_{D_{\mathcal{S}}}$ except the set $F_{\mathcal{S}}$.
Check also FIGURE \ref{fig:2}.
\begin{figure}
	\includegraphics[width=80mm,height=80mm]{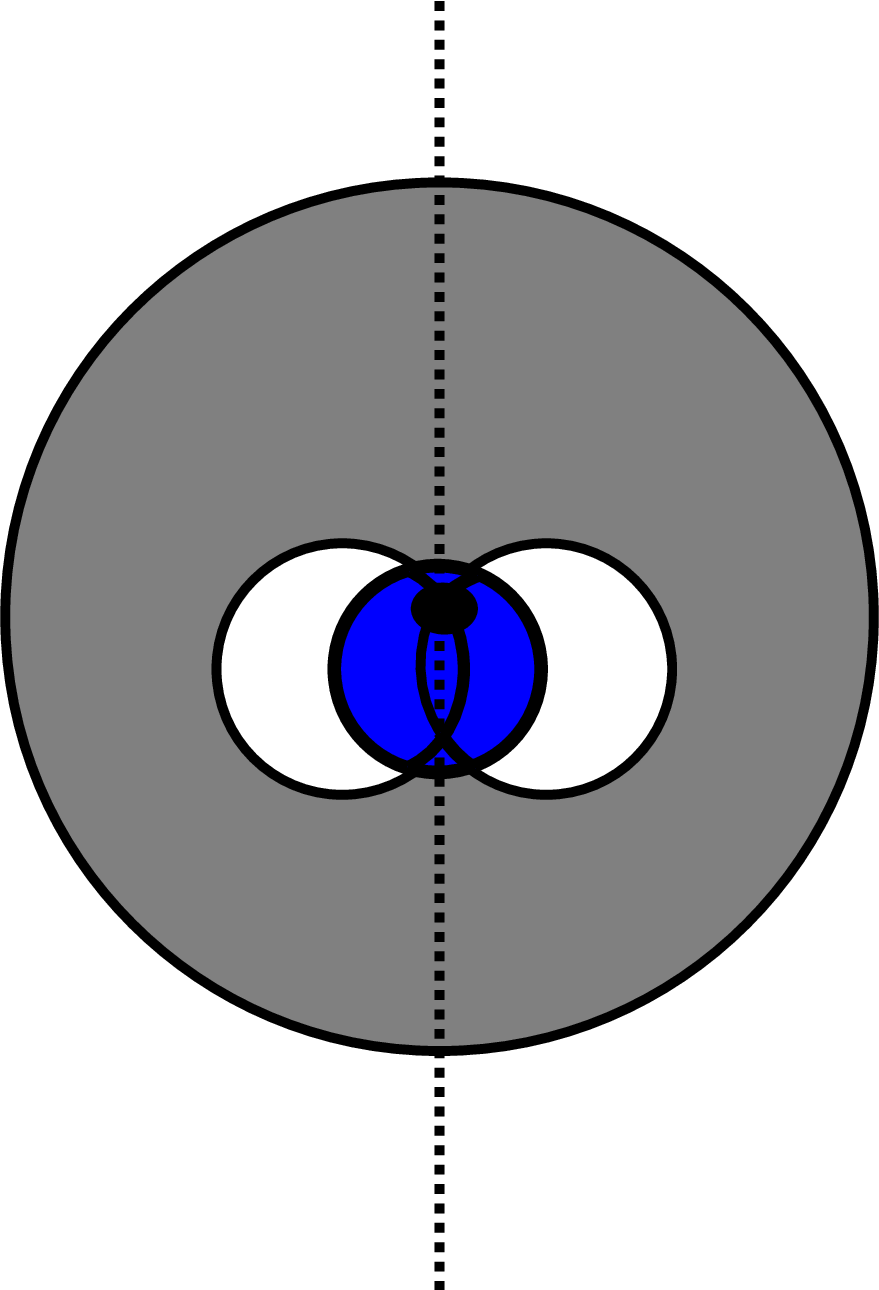}
	\caption{The family $\mathcal{S}$ of circles consists of exactly three circles.
Two of the circles are of a same radius and centered at two points in a line parallel to the horizontal line $\mathbb{R} \times \{0\}$. 
 The disk $D_{j^{\prime}}$ with a point $x_{j^{\prime}} \in D_{j^{\prime}}$ contained in the two circles from $\mathcal{S}$ is added. The region $D_{\mathcal{S}}$ ($D_{\mathcal{S}} \bigcap ({\mathbb{R}}^2-\overline{D_{j^{\prime}}})$) is colored in gray and the disk $D_{j^{\prime}}$ is colored in blue. The dotted line is parallel to the vertical line $\{0\} \times \mathbb{R} \subset {\mathbb{R}}^2$.}
	\label{fig:2}
\end{figure}
We have checked that the assumption on the number $l \geq 2$ is optimal.
This completes our proof.

\end{proof}
We present another simple case related to Theorems \ref{thm:1} and \ref{thm:2}.
\begin{Ex}
	\label{ex:2}
We consider a region $D_{\mathcal{S}}$ surrounded by two circles from $\mathcal{S}$, consisting of the two circles, centered at a same point $x_{0}$, as Example \ref{ex:1}. We can choose an ellipsoid $D_{j^{\prime}}$ centered at $x_{0}$ for (the additional statement of) Theorem \ref{thm:1}: $x_{j^{\prime}}=({x_{j^{\prime}}}_1,{x_{j^{\prime}}}_2) \in D_{j^{\prime}} \bigcap (\overline{D_{\mathcal{S}}}-D_{\mathcal{S}})$ and 
$x_{j^{\prime},1}:=(x_{j^{\prime},1,1},x_{j^{\prime},1,2}) \in D_{j^{\prime}} \bigcap (\overline{D_{\mathcal{S}}}-D_{\mathcal{S}})$ with $x_{j^{\prime},0}:=(\frac{{x_{j^{\prime}}}_1+x_{j^{\prime},1,1}}{2},\frac{{x_{j^{\prime}}}_2+x_{j^{\prime},1,2}}{2})=x_0 \in {\mathbb{R}}^2$ where $x_{j^{\prime}}$ and 
$x_{j^{\prime},1}$ are on the inner circle of $\mathcal{S}$. The pointed set $(D_{j^{\prime}},x_{j^{\prime}})$ with $x_{j^{\prime}}=({x_{j^{\prime}}}_1,{x_{j^{\prime}}}_2) \in D_{j^{\prime}} \bigcap (\overline{D_{\mathcal{S}}}-D_{\mathcal{S}})$ is $(\mathcal{S},D_{\mathcal{S}},F_{D_{\mathcal{S}},1} \bigcup F_{D_{\mathcal{S}},2})$-LS and it is not $(\mathcal{S},D_{\mathcal{S}},F_{D_{\mathcal{S}},1} \bigcup F_{D_{\mathcal{S}},2})$-PLS. See also FIGURE \ref{fig:3}.
\begin{figure}
	\includegraphics[width=80mm,height=80mm]{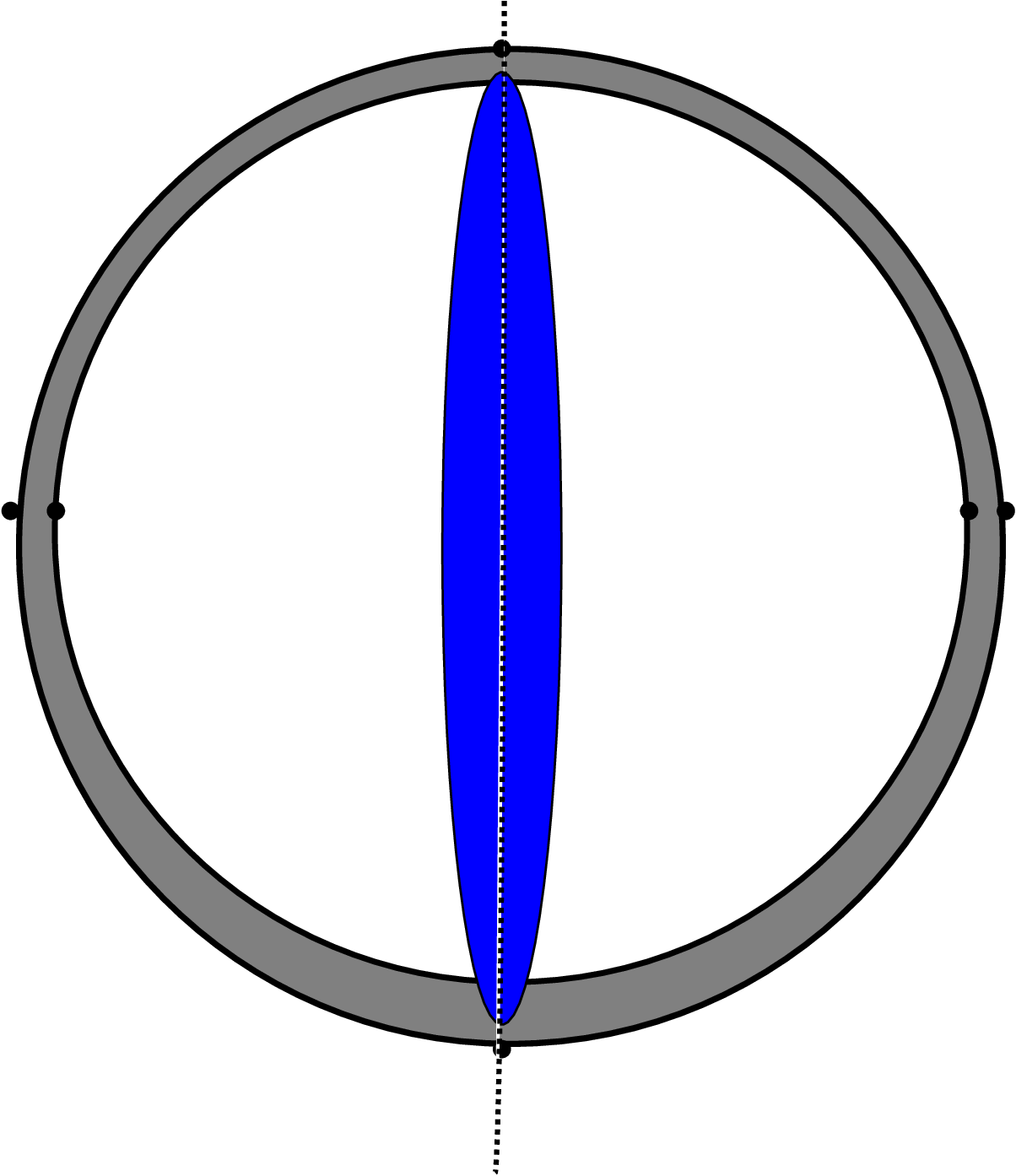}
	\caption{A region surrounded by two concentric circles and an ellipsoid $D_{j^{\prime}}$, colored in blue, and having the black colored boundary. This is for (the additional statement of) Theorem \ref{thm:1}.}
	\label{fig:3}
\end{figure}

We can also see that in Theorem \ref{thm:2}, considering only circles for $\mathcal{S}$ and only closed disks, whose boundaries are circles, for open connected sets $D_{j^{\prime}}$, is essential.
\end{Ex}

In our general arguments on Definition \ref{def:4}, as a natural condition stronger than the condition that the intersection $D_{j^{\prime}} \bigcap  \overline{D_{\mathcal{S}}}$ is connected, we can consider the condition that not only this intersection, but also the intersection $D_{j^{\prime}} \bigcap  (\overline{D_{\mathcal{S}}}-D_{\mathcal{S}})$ is connected. It is easy to check that the connectedness of $D_{j^{\prime}} \bigcap  (\overline{D_{\mathcal{S}}}-D_{\mathcal{S}})$ implies the connectedness of $D_{j^{\prime}} \bigcap  \overline{D_{\mathcal{S}}}$.

In the present case, the families $\mathcal{S}$ and ${\mathcal{S}}^{\prime}$ of Definition \ref{def:4} consist of circles and we have the following for this.
\begin{Thm}
	\label{thm:3}
\begin{enumerate}
\item \label{thm:3.1} Let the family $\mathcal{S}$ in Definition \ref{def:4} contain no curve which is a connected component of the zero set of a real polynomial of degree $1$, or a straight line. If the pointed set $(D_{j^{\prime}},x_{j^{\prime}})$ with $x_{j^{\prime}}=({x_{j^{\prime}}}_1,{x_{j^{\prime}}}_2) \in D_{j^{\prime}} \bigcap (\overline{D_{\mathcal{S}}}-D_{\mathcal{S}})$ is $(\mathcal{S},D_{\mathcal{S}},A_{D_{\mathcal{S}}})$-LS, and the set $D_{j^{\prime}} \bigcap  (\overline{D_{\mathcal{S}}}-D_{\mathcal{S}})$ is connected, then $D_{j^{\prime}} \bigcap  (\overline{D_{\mathcal{S}}}-D_{\mathcal{S}})$ must be a piecewise smooth curve containing at most $1$ point contained in $A_{D_{\mathcal{S}}}$. Especially, if the pointed set $(D_{j^{\prime}},x_{j^{\prime}})$ with $x_{j^{\prime}}=({x_{j^{\prime}}}_1,{x_{j^{\prime}}}_2) \in D_{j^{\prime}} \bigcap (\overline{D_{\mathcal{S}}}-D_{\mathcal{S}})$ is $(\mathcal{S},D_{\mathcal{S}},A_{D_{\mathcal{S}}})$-LS, $x_{j^{\prime}}$ is not a point of $A_{D_{\mathcal{S}}}$, and the set $D_{j^{\prime}} \bigcap  (\overline{D_{\mathcal{S}}}-D_{\mathcal{S}})$ is connected, then $D_{j^{\prime}} \bigcap  (\overline{D_{\mathcal{S}}}-D_{\mathcal{S}})$ must be a smooth curve  containing no point of $A_{D_{\mathcal{S}}}$. 
\item \label{thm:3.2} In addition, let the families $\mathcal{S}$ and ${\mathcal{S}}^{\prime}$ in Definition \ref{def:4} consist of circles and $A_{D_{\mathcal{S}}}:=F_{D_{\mathcal{S}},1} \bigcup F_{D_{\mathcal{S}},2}$.
Furthermore, if the pointed set $(D_{j^{\prime}},x_{j^{\prime}})$ with $x_{j^{\prime}}=({x_{j^{\prime}}}_1,{x_{j^{\prime}}}_2) \in D_{j^{\prime}} \bigcap (\overline{D_{\mathcal{S}}}-D_{\mathcal{S}})$ is $(\mathcal{S},D_{\mathcal{S}},A_{D_{\mathcal{S}}})$-S, and the set $D_{j^{\prime}} \bigcap  (\overline{D_{\mathcal{S}}}-D_{\mathcal{S}})$ is connected, in this situation, then we have the following.
\begin{enumerate}
\item \label{thm:3.2.1}
Let $x_{j^{\prime}}=({x_{j^{\prime}}}_1,{x_{j^{\prime}}}_2) \in {D_{j^{\prime}}}_t \bigcap (\overline{D_{\mathcal{S}}}-D_{\mathcal{S}})$ be a point contained in at most one circle from $\mathcal{S}$.
We have a family $\{({D_{j^{\prime}}}_t,x_{j^{\prime}})\}_{t \in [0,\infty)}$ of pointed sets  
parametrized by the real numbers $t \geq 0$ satisfying the following. 
\begin{enumerate}
\item
\label{thm:3.2.1.1} The sets ${D_{j^{\prime}}}_t$ are disks bounded by the circles $\overline{{D_{j^{\prime}}}_t}-{D_{j^{\prime}}}_t$.
\item \label{thm:3.2.1.2} The pointed set $({D_{j^{\prime}}}_t,x_{j^{\prime}})$ is $(\mathcal{S},D_{\mathcal{S}},A_{D_{\mathcal{S}}})$-S, and the set ${D_{j^{\prime}}}_t \bigcap (\overline{D_{\mathcal{S}}}-D_{\mathcal{S}})$ is connected. Let $({\mathcal{S}}_t,D_{{\mathcal{S}}_t})$ denote the resulting real algebraic domain obtained by putting ${D_{j^{\prime}}}_t$. 
\item \label{thm:3.2.1.3} The embeddings of the disks ${D_{j^{\prime}}}_t$ give a smooth isotopy with ${D_{j^{\prime}}}_0=D_{j^{\prime}}$.
\item \label{thm:3.2.1.4} The radii of ${D_{j^{\prime}}}_t$ converge to $0$ as $t$ increases.
\item \label{thm:3.2.1.5} For each $i=1,2$, the Poincar\'e-Reeb V-digraphs of $({\mathcal{S}}_t,D_{{\mathcal{S}}_t})$ for ${\pi}_{2,1,i}$ are mutually isomorphic for all $t$.
\end{enumerate} 
\item \label{thm:3.2.2}
We have a family $\{({D_{j^{\prime}}}_t,x_{j^{\prime}})\}_{t \in [0,\infty)}$ of pointed sets with $x_{j^{\prime}}=({x_{j^{\prime}}}_1,{x_{j^{\prime}}}_2) \in {D_{j^{\prime}}}_t \bigcap (\overline{D_{\mathcal{S}}}-D_{\mathcal{S}})$ parametrized by the real numbers $t \geq 0$ satisfying the following.
\begin{enumerate}
\item \label{thm:3.2.2.2} The sets ${D_{j^{\prime}}}_t$ are disks bounded by the circles $\overline{{D_{j^{\prime}}}_t}-{D_{j^{\prime}}}_t$ and centered at $x_{j^{\prime}}$.
\item The conditions {\rm (}\ref{thm:3.2.1.2}{\rm )},  {\rm (}\ref{thm:3.2.1.3}{\rm )}, {\rm (}\ref{thm:3.2.1.4}{\rm )}, and {\rm (}\ref{thm:3.2.1.5}{\rm )}.
\end{enumerate} 
\end{enumerate}
\end{enumerate}
	\end{Thm}

\begin{proof}
We first prove (\ref{thm:3.1}).

Suppose that the pointed set $(D_{j^{\prime}},x_{j^{\prime}})$ with $x_{j^{\prime}}=({x_{j^{\prime}}}_1,{x_{j^{\prime}}}_2) \in D_{j^{\prime}} \bigcap (\overline{D_{\mathcal{S}}}-D_{\mathcal{S}})$ is $(\mathcal{S},D_{\mathcal{S}},A_{D_{\mathcal{S}}})$-LS, that the set $D_{j^{\prime}} \bigcap  (\overline{D_{\mathcal{S}}}-D_{\mathcal{S}})$ is connected, and that $D_{j^{\prime}} \bigcap  (\overline{D_{\mathcal{S}}}-D_{\mathcal{S}})$ is a piecewise smooth curve containing two distinct points $p_1$ and $p_2$ contained in $F_{D_{\mathcal{S}},1} \bigcup F_{D_{\mathcal{S}},2} \subset A_{D_{\mathcal{S}}}$.
Here, $p_1$ and $p_2$ are also in the interior of $D_{j^{\prime}}$. In addition, the values of the 1st components of $p_1$ and $p_2$ or the values of the 2nd components of them are same, by the assumption that the pointed set $(D_{j^{\prime}},x_{j^{\prime}})$ with $x_{j^{\prime}}=({x_{j^{\prime}}}_1,{x_{j^{\prime}}}_2) \in D_{j^{\prime}} \bigcap (\overline{D_{\mathcal{S}}}-D_{\mathcal{S}})$ is $(\mathcal{S},D_{\mathcal{S}},A_{D_{\mathcal{S}}})$-LS.  
Furthermore, $D_{j^{\prime}} \bigcap  (\overline{D_{\mathcal{S}}}-D_{\mathcal{S}})$ must contain a smooth curve connecting two points $p_1$ and $p_2$. By the assumption that $\mathcal{S}$ has no straight line, we have another point $p_3 \in D_{j^{\prime}} \bigcap  (\overline{D_{\mathcal{S}}}-D_{\mathcal{S}})$ being also a point of $A_{D_{\mathcal{S}}}$ and the following hold: for $j=1,2$, the values of the 1st components of $p_j$ and $p_3$ and the values of the 2nd components of them are different. This contradicts the assumption that the pointed set $(D_{j^{\prime}},x_{j^{\prime}})$ is $(\mathcal{S},D_{\mathcal{S}},A_{D_{\mathcal{S}}})$-LS.

We have shown that $D_{j^{\prime}} \bigcap  (\overline{D_{\mathcal{S}}}-D_{\mathcal{S}})$ must be a piecewise smooth curve containing at most $1$ point contained in exactly two circles from $\mathcal{S}$ under the assumption that the pointed set $(D_{j^{\prime}},x_{j^{\prime}})$ with $x_{j^{\prime}}=({x_{j^{\prime}}}_1,{x_{j^{\prime}}}_2) \in D_{j^{\prime}} \bigcap (\overline{D_{\mathcal{S}}}-D_{\mathcal{S}})$ is $(\mathcal{S},D_{\mathcal{S}},A_{D_{\mathcal{S}}})$-S and that the set $D_{j^{\prime}} \bigcap  (\overline{D_{\mathcal{S}}}-D_{\mathcal{S}})$ is connected.

Especially, in addition, if $x_{j^{\prime}}$ is not a point of $A_{D_{\mathcal{S}}}$, then $D_{j^{\prime}} \bigcap  (\overline{D_{\mathcal{S}}}-D_{\mathcal{S}})$ must be a smooth curve containing no point of $A_{D_{\mathcal{S}}}$ and we can show this by a similar argument. 


We prove (\ref{thm:3.2.1}). 

In short, by (\ref{thm:3.1}), for the families $\mathcal{S}$ and ${\mathcal{S}}^{\prime}$ in Definition \ref{def:4} consisting of circles with $A_{D_{\mathcal{S}}}:=F_{D_{\mathcal{S}},1} \bigcup F_{D_{\mathcal{S}},2}$, $D_{j^{\prime}}$ must not contain any point of $A_{D_{\mathcal{S}}}$ except $x_{j^{\prime}} \in D_{j^{\prime}}$.

We consider the uniquely obtained straight segment connecting the two points of $(\overline{D_{j^{\prime}}}-D_{j^{\prime}}) \bigcap (\overline{D_{\mathcal{S}}}-D_{\mathcal{S}})$. We move the straight line containing the segment to the unique straight line containing $x_{j^{\prime}}$ and parallel to the original straight line continuously and smoothly. We also move the lines in such a way that these straight lines are mutually parallel. 
Thus we have a continuous and smooth family $\{L_{j^{\prime},t}\}$ of such straight lines parametrized by non-negative numbers $t \geq 0$: of course $L_{j^{\prime},0}$ is the initial straight line here.
We have our desired family $\{{D_{j^{\prime}}}_t\}$ of disks satisfying the following. We also present several important arguments.
\begin{itemize}
 \item This is for the mutually parallel straight lines $L_{j^{\prime},t}$. These lines converge to the straight line containing $x_{j^{\prime}}$ and parallel to the original straight line $L_{j^{\prime},0}$ as the numbers $t$ increase to the infinity $\infty$.
\item 
The set $(\overline{{D_{j^{\prime}}}_t}-{D_{j^{\prime}}}_t) \bigcap (\overline{D_{\mathcal{S}}}-D_{\mathcal{S}})$ is a two-point set and contained in the straight line $L_{j^{\prime},t}$. Let the $1$-dimensional set represented as the union of the circle $\overline{{D_{j^{\prime}}}_t}-{D_{j^{\prime}}}_t$ and the unique straight segment connecting the two points of the two-point set $(\overline{{D_{j^{\prime}}}_t}-{D_{j^{\prime}}}_t) \bigcap (\overline{D_{\mathcal{S}}}-D_{\mathcal{S}}) \subset L_{j^{\prime},t}$, be denoted by ${S_{j^{\prime}}}_{t,\mathcal{S}}$. 

Here, we need the notion of similarity of subsets in the plane ${\mathbb{R}}^2$. For distinct numbers $t_1, t_2 \geq 0$, ${S_{j^{\prime}}}_{t_1,\mathcal{S}}$ and ${S_{j^{\prime}}}_{t_2,\mathcal{S}}$ are similar. Furthermore, we do not need to use any orientation reversing transformation to transform ${S_{j^{\prime}}}_{t_1,\mathcal{S}}$ to ${S_{j^{\prime}}}_{t_2,\mathcal{S}}$ by a composition of transformations preserving the similarity, in ${\mathbb{R}}^2$.

Here, the assumption that $x_{j^{\prime}}$ is a point contained in at most one circle from $\mathcal{S}$ implies that $x_{j^{\prime}} \in {D_{j^{\prime}}}_t \bigcap (\overline{D_{\mathcal{S}}}-D_{\mathcal{S}})$.

\item The continuous and smooth deformation increasing $t$ induces isomorphisms between the Poincar\'e-Reeb V-digraphs of $({\mathcal{S}}_t,D_{{\mathcal{S}}_t})$ for ${\pi}_{2,1,i}$. This is easily shown by investigating tangent vectors of the circles at points there.

We explain this more precisely. By slightly changing the values of $t$, we can induce the isomorphisms. The compactness of the closed interval $[0,t]:=\{0 \leq t^{\prime} \leq t\}$ for an arbitrary positive number $t>0$ with the assumption that the pointed set $({D_{j^{\prime}}}_t,x_{j^{\prime}})$ is $(\mathcal{S},D_{\mathcal{S}},A_{D_{\mathcal{S}}})$-S completes the proof of this.
\end{itemize}

For (\ref{thm:3.2.2}), we have a continuous and smooth family of disks as (\ref{thm:3.2.1}) easily and canonically. We can check the existence of the induced isomorphisms between the Poincar\'e-Reeb V-digraphs of $({\mathcal{S}}_t,D_{{\mathcal{S}}_t})$ for ${\pi}_{2,1,i}$.

This completes the proof.
\end{proof}
The following shows a kind of our additional remark on Theorem \ref{thm:3}.
\begin{Rem}
\label{rem:1}
Theorem \ref{thm:3} (\ref{thm:3.2.1}) respects a rule of putting disks one after another, as in \cite{kitazawa8}, where we consider the situation of Definition \ref{def:4}. Theorem \ref{thm:3} (\ref{thm:3.2.2}) respects a rule of putting disks as in \cite{kitazawa5}. \cite{kitazawa5} shows an explicit rule of adding disks one after another, first. \cite{kitazawa8}, especially, \cite[Theorem 2 and Theorem 3]{kitazawa8}, explains that our rules in \cite{kitazawa5, kitazawa6} are generalized to a certain rule. Note that \cite[Theorem 2 and Theorem 3]{kitazawa8} only discusses the case the point $x_{j^{\prime}} \in D_{j^{\prime}}$ is not a point of $A_{D_{\mathcal{S}}}:=F_{D_{\mathcal{S}},1} \bigcup F_{D_{\mathcal{S}},2}$

Note also that our study is related to explicit construction of a natural real algebraic map onto the region $D_{\mathcal{S}}$ surrounded by the real algebraic curves of $\mathcal{S}$, started in \cite{kitazawa1} and followed by \cite{kitazawa2} by the author. 
See also Theorem \ref{thm:6}, presented in the end, and related exposition there.
\cite{kitazawa5, kitazawa6} respect \cite{kitazawa2}.

\end{Rem}
\section{Additional related observations.}
We consider cases from \cite[Theorem 1]{kitazawa7}. We abuse our notation before or naturally defined one.
\begin{Thm}[\cite{kitazawa7}]
	\label{thm:4}
	Let $c_G:G \rightarrow \mathbb{R}$ be a piecewise smooth function with the following.
	\begin{itemize}
		\item The restriction of $c_G$ to each edge is injective.
		\item The degree of each vertex is not $2$.
		\item If the function $c_G$ has a local extremum at a vertex $v$, then $v$ is of degree $1$.
		\item The function $c_G$ is the composition of a piecewise smooth embedding $\tilde{c_G}:G \rightarrow {\mathbb{R}}^2$ with ${\pi}_{2,1,1}$.
	\end{itemize} 
	Then we have a refined real algebraic domain $({\mathcal{S}}_{c_G},D_{{\mathcal{S}}_{c_G}},F_{D_{{\mathcal{S}}_{c_G}},1} \bigcup F_{D_{{\mathcal{S}}_{c_G}},2})$ with poles enjoying the following.
	\begin{enumerate}
		\item We have a piecewise smooth homeomorphism $\phi:G \rightarrow W_{D_{{\mathcal{S}}_{c_G}},1}$ mapping the vertex set of $G$ into the vertex set of the Poincar\'e-Reeb V-digraph $(W_{D_{{\mathcal{S}}_{c_G}},1},V_{D_{{\mathcal{S}}_{c_G}},1})$.
		\item The vertex $v$ of the graph $G$ is mapped to $\phi(v)$ with the constraint $c_G(v)=(V_{D_{{\mathcal{S}}_{c_G}},1} \circ \phi)(v)$
	\end{enumerate}
	\end{Thm}
	Note that the finite set $F_{D_{{\mathcal{S}}_{c_G}},1} \bigcup F_{D_{{\mathcal{S}}_{c_G}},2}$ is not essential or respected in \cite{kitazawa7}.
\begin{proof}[Reviewing the original proof]
We review our proof shortly. Check also our original proof. In the original study \cite{kitazawa7}, we can also check figures for our sketch of the proof there.

The most important is "approximation by real polynomials respecting the derivatives of degree $d \leq 2$". See \cite{bochnakcosteroy, kollar, lellis} for this and related survey on real algebraic geometry and \cite{bodinpopescupampusorea, elredge} for explicit usage. 

We can have the zero set $S_{0,c_G}$ of a refined real algebraic region\\ $(\{S_{0,c_G}\},D_{\{S_{0,c_G}\}},F_{D_{\{S_{0,c_G}\}},1} \bigcup F_{D_{\{S_{0,c_G}\}},2})$
with poles. In addition, there exist a sufficiently small positive number ${\epsilon}_0$, another sufficiently small positive one ${\epsilon}>0$, and another sufficiently small positive one $0<{\epsilon}^{\prime}<\epsilon$. For these objects and numbers, we can do so that the following properties are enjoyed, in addition.
\begin{itemize}
	\item The closure of $D_{\{S_{0,c_G}\}}$ is regarded as a regular neighborhood of the image of $\tilde{c_G}:G \rightarrow {\mathbb{R}}^2$ with ${\pi}_{2,1,1}$ where $\tilde{c_G}$ can be changed suitably and is changed from the original one in some suitable way. \cite{hirsch} discusses regular neighborhoods in the (piecewise) smooth category. More precisely, the distance of any point in the boundary $S_{0,c_G}$ of the closure of $D_{\{S_{0,c_G}\}}$ and the graph $\tilde{c_G}(G)$ is smaller than the sufficiently small positive number ${\epsilon}_0>0$.
	\item The set of all singular points of the restriction of ${\pi}_{2,1,1}$ to $S_{0,c_G}$ is $F_{D_{\{S_{0,c_G}\}},1}$. The function is also a so-called {\it Morse} function at each of these points: in other words the value of the 2nd derivative there is not $0$ (for any local coordinate). We have this from our definition together with approximation from elementary singularity theory and real algebraic geometry.
	\item For each point of $F_{D_{\{S_{0,c_G}\}},1}$, a vertex of $v$ of $\tilde{c_G}(G)$ is canonically corresponded. 
\item For each vertex $v \in \tilde{c_G}(G)$ of the graph which is not of degree $1$, let $n_{v,-}$ ($n_{v,+}$) denote the number of oriented edges of the V-digraph $(\tilde{c_G}(G),{\pi}_{2,1,1} {\mid}_{V_{\tilde{c_G}(G)}})$ entering (resp. departing from) $v$: let $V_{\tilde{c_G}(G)}$ denote the vertex set of the graph. For exactly $n_{v,-}+n_{v,+}$ of singular points of the restriction of ${\pi}_{2,1,1}$ to $S_{0,c_G}$, $v$ is corresponded. The $j$-th point of the $n_{v,-}$ points is of the form $({\pi}_{2,1,1}(v)-{\epsilon}_{v,-,j},a_{v,-,j})$ and the $j$-th point of the $n_{v,+}$ points is of the form $({\pi}_{2,1,1}(v)+{\epsilon}_{v,+,j},a_{v,+,j})$ under the following rule.
	\begin{itemize}
		\item The relations $0<{\epsilon}_{v,-,j},{\epsilon}_{v,+,j} <{\epsilon}^{\prime}<\epsilon$ are satisfied where the numbers ${\epsilon}_{v,-,j_1}$ and ${\epsilon}_{v,-,j_2}$ (${\epsilon}_{v,+,j_1}$ and ${\epsilon}_{v,+,j_2}$) are mutually disjoint for distinct numbers $j_1$ and $j_2$ and sufficiently small.
		\item For each $v$ of the graph, a real number $a_v$ is chosen in such a way that for distinct vertices of the graph the values are distinct. We can also do in such a way that the relation $a_{v,-,j_1}<a_{v,-,j_2}$ holds for an arbitrary pair $(j_1,j_2)$ with $j_1<j_2$, that the relation $a_{v,+,j_1}<a_{v,+,j_2}$ holds for an arbitrary pair $(j_1,j_2)$ with $j_1<j_2$, and that the relations $a_{v,-,j_{-}}<a_{v,+,j_{+}}$ and $a_{v}-{\epsilon}^{\prime}<a_{v,-,j_{-}}, a_{v,+,j_{+}}<a_{v}+{\epsilon}^{\prime}$ hold for an arbitrary pair $(j_{-},j_{+})$.
		 \end{itemize}
\item For each vertex $v_0 \in \tilde{c_G}(G)$ of the graph which is of degree $1$, let $n_{0,v_0,-}$ ($n_{0,v_0,+}$) denote the number of oriented edges of the V-digraph entering (resp. departing from) $v_0$. Either $n_{0,v_0,-}$ or $n_{0,v_0,+}$ is $0$ of course.
For exactly $n_{0,v_0,-}+n_{0,v_0,+}$ of singular points of the restriction of ${\pi}_{2,1,1}$ to $S_{0,c_G}$, $v_0$ is corresponded. The $j$-th point of the $n_{0,v_0,-}$ points is of the form $({\pi}_{2,1,1}(v)+{\epsilon}_{v_0,-,j},a_{v_0,-,j})$ and the $j$-th point of the $n_{0,v_0,+}$ points is of the form $({\pi}_{2,1,1}(v)-{\epsilon}_{v_0,+,j},a_{v_0,+,j})$ under the following rule.
	\begin{itemize}
		\item The relations $0<{\epsilon}_{v_0,-,j},{\epsilon}_{v_0,+,j} <\epsilon$ are satisfied where the numbers ${\epsilon}_{v_0,-,j_1}$ and ${\epsilon}_{v_0,-,j_2}$ (${\epsilon}_{v_0,+,j_1}$ and ${\epsilon}_{v_0,+,j_2}$) are mutually disjoint for distinct numbers $j_1$ and $j_2$ and sufficiently small.
		\item For each $v_0$ of the graph, a real number $a_{v_0}$ is chosen in such a way that for distinct vertices of the graph the values are distinct and not equal to any $a_v$ for any vertex of degree at least $3$ of the graph. Furthermore, we can do in such a way that the relation $a_{v_0,-,j_1}<a_{v_0,-,j_2}$ holds for an arbitrary pair $(j_1,j_2)$ with $j_1<j_2$, that the relation $a_{v_0,+,j_1}<a_{v_0,+,j_2}$ holds for an arbitrary pair $(j_1,j_2)$ with $j_1<j_2$, and that the relations $a_{v_0,-,j_{-}}<a_{v_0,+,j_{+}}$ and $a_{v_0}-{\epsilon}^{\prime}<a_{v_0,-,j_{-}}, a_{v_0,+,j_{+}}<a_{v_0}+{\epsilon}^{\prime}$ hold for an arbitrary pair $(j_{-},j_{+})$.
		 \end{itemize}
	\end{itemize}  

We can put an ellipsoid $D_{j^{\prime}}$ of the standard form centered at $({\pi}_{2,1,1}(v) \pm {\epsilon}_{v,\pm,j},a_{v,\pm,j}) \in F_{D_{\mathcal{S}},1}$ and satisfying "$a_1{{\epsilon}_{v,\pm,j}}^2=r$ and $a_2$ being sufficiently large" in the definition. Remember that the relations $0<{\epsilon}_{v,-,j},{\epsilon}_{v,+,j} <{\epsilon}^{\prime}<\epsilon$ are satisfied and that the numbers ${\epsilon}_{v,-,j_1}$ and ${\epsilon}_{v,-,j_2}$ (${\epsilon}_{v,+,j_1}$ and ${\epsilon}_{v,+,j_2}$) are mutually disjoint for distinct numbers $j_1$ and $j_2$ and sufficiently small.
Remember also that the relation $a_{v_0,-,j_1}<a_{v_0,-,j_2}$ holds for an arbitrary pair $(j_1,j_2)$ with $j_1<j_2$, that the relation $a_{v_0,+,j_1}<a_{v_0,+,j_2}$ holds for an arbitrary pair $(j_1,j_2)$ with $j_1<j_2$, and that the relations $a_{v_0,-,j_{-}}<a_{v_0,+,j_{+}}$ and $a_{v_0}-{\epsilon}^{\prime}<a_{v_0,-,j_{-}}, a_{v_0,+,j_{+}}<a_{v_0}+{\epsilon}^{\prime}$ hold for an arbitrary pair $(j_{-},j_{+})$. We can see that each ellipsoid can be chosen one after another and disjointly, for each critical point of the function ${\pi}_{2,1,1} {\mid}_{S_{0,c_G}}$ corresponding to each vertex $v$ of degree at least $3$.

 We can put an ellipsoid of the standard form centered at some point and containing the point $({\pi}_{2,1,1}(v_0)-{\epsilon}_{v_0,+,j},a_{v_0,+,j}) \in F_{D_{\{S_{0,c_G}\}},1}$ ($({\pi}_{2,1,1}(v_0)+{\epsilon}_{v_0,-,j},a_{v_0,-,j}) \in F_{D_{\{S_{0,c_G}\}},1}$) in the interior in such a way that the boundary $\overline{D_{j^{\prime}}}-D_{j^{\prime}}$ of its closure contains a point $p_{j^{\prime}} \in S_{0,c_G}$ with ${\pi}_{2,1,1}(p_{j^{\prime}})={\pi}_{2,1,1}(v_0)$. We also do in such a way that the intersection $(\overline{D_{j^{\prime}}}-D_{j^{\prime}}) \bigcap \overline{D_{\{S_{0,c_G}\}}}$ is mapped by ${\pi}_{2,1,1}$ into $\mathbb{R}$ as an injective function with the minimum (resp. maximum) ${\pi}_{2,1,1}(v_0)$ .
Remember that the relations $0<{\epsilon}_{v_0,-,j},{\epsilon}_{v_0,+,j} <{\epsilon}^{\prime}<\epsilon$ are satisfied and that the numbers ${\epsilon}_{v_0,-,j_1}$ and ${\epsilon}_{v_0,-,j_2}$ (${\epsilon}_{v_0,+,j_1}$ and ${\epsilon}_{v_0,+,j_2}$) are mutually disjoint for distinct numbers $j_1$ and $j_2$ and sufficiently small. Remember also that the relation $a_{v_0,-,j_1}<a_{v_0,-,j_2}$ holds for an arbitrary pair $(j_1,j_2)$ with $j_1<j_2$, that the relation $a_{v_0,+,j_1}<a_{v_0,+,j_2}$ holds for an arbitrary pair $(j_1,j_2)$ with $j_1<j_2$, and that the relations $a_{v_0,-,j_{-}}<a_{v_0,+,j_{+}}$ and $a_{v_0}-{\epsilon}^{\prime}<a_{v_0,-,j_{-}}, a_{v_0,+,j_{+}}<a_{v_0}+{\epsilon}^{\prime}$ hold for an arbitrary pair $(j_{-},j_{+})$. We can see that each ellipsoid can be chosen one after another, disjointly, and disjoint from the previous ellipsoids, for each critical point of the function ${\pi}_{2,1,1} {\mid}_{S_{0,c_G}}$ corresponding to each vertex $v_0$ of degree $1$.

We put ellipsoids $D_{j^{\prime}}$ of the standard form one after another to have our desired result.
\end{proof} 
\begin{Prop}
\label{prop:1}
	In the present inductive procedure, for a step putting an ellipsoid $D_{j^{\prime}}$ centered at the point $x_{j^{\prime}}:=({\pi}_{2,1,1}(v) \pm {\epsilon}_{v,\pm,j},a_{v,\pm,j})$
or 
$x_{j^{\prime}}:=({\pi}_{2,1,1}(v_0) \pm {\epsilon}_{v_0,\pm,j},a_{v_0,\pm,j})$, we abuse the notation from Definition \ref{def:4}.

First, $D_{j^{\prime}}$ is $(\mathcal{S},D_{\mathcal{S}})$-connected.

In addition, we assume at least one of the following in a step in the procedure.
\begin{itemize}
\item At least one ellipsoid centered at a point of the form $({\pi}_{2,1,1}(v) \pm {\epsilon}_{v,\pm,j_0},a_{v,\pm,j_0})$ has been put in the case $v$ is of degree at least $3$.
\item At least one ellipsoid centered at a point of the form $({\pi}_{2,1,1}(v_0) \pm {\epsilon}_{v_0,\pm,j_0},a_{v_0,\pm,j_0})$ has been put in the case $v_0$ is of degree $1$.
\end{itemize}
 
	Under this additional assumption, the pointed set $(D_j,x_{j^{\prime}})$ is $(\mathcal{S},D_{\mathcal{S}},A_{D_{\mathcal{S}}})$-PLS and it is not $(\mathcal{S},D_{\mathcal{S}},A_{D_{\mathcal{S}}})$-PS.
\end{Prop}
\begin{proof}
The first part follows from assumptions on a sufficiently small positive number ${\epsilon}_0>0$, another sufficiently small one $\epsilon>0$, and sufficiently small ones with the relations $0<{\epsilon}_{v,-,j}, {\epsilon}_{v,+,j}, {\epsilon}_{v_0,-,j},{\epsilon}_{v_0,+,j}<{\epsilon}^{\prime}<\epsilon$.

We see the additional part. The set
$\overline{D_{j^{\prime}}}$ contains a point of the boundary $\overline{D_{j^{\prime}}}-D_{j^{\prime}}$ such that the value of the projection ${\pi}_{2,1,1}$ there is ${\pi}_{2,1,1}(v)$ or ${\pi}_{2,1,1}(v_0)$. The set
$\overline{D_{j^{\prime}}}$ contains no point from $F_{D_{\mathcal{S}},2}$ by considering the local shape of the curve $S_{0,c_G}$. From the additional assumption, there exists at least one point different from this such that the value of the projection ${\pi}_{2,1,1}$ there is ${\pi}_{2,1,1}(v)$ or ${\pi}_{2,1,1}(v_0)$ in $F_{D_{\mathcal{S}},1}$. We can easily check that the pointed set $(D_j,x_{j^{\prime}})$ is $(\mathcal{S},D_{\mathcal{S}},A_{D_{\mathcal{S}}})$-PLS and it is not $(\mathcal{S},D_{\mathcal{S}},A_{D_{\mathcal{S}}})$-PS.
\end{proof}
We present another example as our new result.
\begin{Thm}
\label{thm:5}
We also abuse the notation from Definition \ref{def:4} here.
In our steps here, for each pair of ellipsoids of the standard form centered at $({\pi}_{2,1,1}(v_0)-{\epsilon}_{v_0,+,j},a_{v_0,+,j})$ and $({\pi}_{2,1,1}(v_0)-{\epsilon}_{v_0,+,j+1},a_{v_0,+,j+1})$ such that ${\pi}_{2,1,1}(v_0)$ is the minimum of the function $c_G$ or ones of the standard form centered at $({\pi}_{2,1,1}(v_0)+{\epsilon}_{v_0,-,j},a_{v_0,-,j})$ and $({\pi}_{2,1,1}(v_0)+{\epsilon}_{v_0,-,j+1},a_{v_0,-,j+1})$ such that ${\pi}_{2,1,1}(v_0)$ is the maximum of the function $c_G$, we can replace the two steps into the step as follows to have the result same as Theorem \ref{thm:4}.
\begin{enumerate}
\item \label{thm:5.1} This is for the former case. We consider the minimum $\min\{{\pi}_{2,1,1}(v_0)-{\epsilon}_{v_0,+,j},{\pi}_{2,1,1}(v_0)-{\epsilon}_{v_0,+,j+1}\}$ in the two-element set $\{{\pi}_{2,1,1}(v_0)-{\epsilon}_{v_0,+,j},{\pi}_{2,1,1}(v_0)-{\epsilon}_{v_0,+,j+1}\}$. For a suitable real number $a_{v_0,+,j}<a_{v_0,+,j,j+1}<a_{v_0,+,j+1}$, we put 
the set $\overline{D_{j^{\prime}}}$ defined as an ellipsoid $D_{j^{\prime}}$ of the standard form centered at a point of the form $(\min \{{\pi}_{2,1,1}(v_0)-{\epsilon}_{v_0,+,j},{\pi}_{2,1,1}(v_0)-{\epsilon}_{v_0,+,j+1}\},a_{v_0,+,j,j+1})$ with the following properties.
\begin{enumerate}
\item \label{thm:5.1.1} The boundary $\overline{D_{j^{\prime}}}-D_{j^{\prime}}$ contains two points $p_{j^{\prime},j^{\prime \prime}} \in S_{0,c_G}$ with ${\pi}_{2,1,1}(p_{j^{\prime}},j^{\prime \prime})={\pi}_{2,1,1}(v_0)$ {\rm (}$j^{\prime \prime}=1,2${\rm )}. 
\item \label{thm:5.1.2} The intersection $D_{j^{\prime}} \bigcap \overline{D_{\mathcal{S}}}$ consists of exactly two connected components and contains the points $({\pi}_{2,1,1}(v_0)-{\epsilon}_{v_0,+,j},a_{v_0,+,j}) \in F_{D_{\mathcal{S}},1} \subset \overline{D_{\mathcal{S}}}-D_{\mathcal{S}}$ and $({\pi}_{2,1,1}(v_0)-{\epsilon}_{v_0,+,j+1},a_{v_0,+,j+1}) \in F_{D_{\mathcal{S}},1} \subset \overline{D_{\mathcal{S}}}-D_{\mathcal{S}}$.
\item \label{thm:5.1.3}
 The set $(\overline{D_{j^{\prime}}}-D_{j^{\prime}}) \bigcap \overline{D_{\mathcal{S}}}$ contains exactly two connected components and each of the two connected components is mapped by ${\pi}_{2,1,1}$ into $\mathbb{R}$ as an injective function with the minimum ${\pi}_{2,1,1}(v_0)$.
\end{enumerate}
\item \label{thm:5.2} This is for the latter case. We consider the maximum $\max\{{\pi}_{2,1,1}(v_0)+{\epsilon}_{v_0,-,j},{\pi}_{2,1,1}(v_0)+{\epsilon}_{v_0,-,j+1}\}$ in the two-element set $\{{\pi}_{2,1,1}(v_0)+{\epsilon}_{v_0,-,j},{\pi}_{2,1,1}(v_0)+{\epsilon}_{v_0,-,j+1}\}$. For a suitable real number $a_{v_0,-,j}<a_{v_0,-,j,j+1}<a_{v_0,-,j+1}$, we put 
the set $\overline{D_{j^{\prime}}}$ defined as an ellipsoid $D_{j^{\prime}}$ of the standard form centered at a point of the form $(\max \{{\pi}_{2,1,1}(v_0)+{\epsilon}_{v_0,-,j},{\pi}_{2,1,1}(v_0)+{\epsilon}_{v_0,-,j+1}\},a_{v_0,-,j,j+1})$ with the following properties.
\begin{enumerate}
\item
\label{thm:5.2.1}
 The boundary $\overline{D_{j^{\prime}}}-D_{j^{\prime}}$ contains two points $p_{j^{\prime},j^{\prime \prime}} \in S_{0,c_G}$ with ${\pi}_{2,1,1}(p_{j^{\prime}},j^{\prime \prime})={\pi}_{2,1,1}(v_0)$ {\rm (}$j^{\prime \prime}=1,2${\rm )}.
\item 
\label{thm:5.2.2}
The intersection $D_{j^{\prime}} \bigcap \overline{D_{\mathcal{S}}}$ consists of exactly two connected components and contains the points $({\pi}_{2,1,1}(v_0)+{\epsilon}_{v_0,-,j},a_{v_0,-,j}) \in F_{D_{\mathcal{S}},1} \subset  \overline{D_{\mathcal{S}}}-D_{\mathcal{S}}$ and $({\pi}_{2,1,1}(v_0)+{\epsilon}_{v_0,-,j+1},a_{v_0,-,j+1}) \in F_{D_{\mathcal{S}},1} \subset \overline{D_{\mathcal{S}}}-D_{\mathcal{S}}$.
\item
\label{thm:5.2.3}
 The set $(\overline{D_{j^{\prime}}}-D_{j^{\prime}}) \bigcap \overline{D_{\mathcal{S}}}$ contains exactly two connected components and each of the two connected components is mapped by ${\pi}_{2,1,1}$ into $\mathbb{R}$ as an injective function with the maximum ${\pi}_{2,1,1}(v_0)$.
\end{enumerate}
\item \label{thm:5.3} The pointed set $(D_j,x_{j^{\prime}})$ is not $(\mathcal{S},D_{\mathcal{S}},A_{D_{\mathcal{S}}})$-LS.
\end{enumerate}
\end{Thm}
\begin{proof}
We prove (\ref{thm:5.1}).

Remember again that the relations $0<{\epsilon}_{v_0,-,j},{\epsilon}_{v_0,+,j} <{\epsilon}^{\prime}<\epsilon$ are satisfied and that the numbers ${\epsilon}_{v_0,-,j_1}$ and ${\epsilon}_{v_0,-,j_2}$ (${\epsilon}_{v_0,+,j_1}$ and ${\epsilon}_{v_0,+,j_2}$) are mutually disjoint for distinct numbers $j_1$ and $j_2$ and sufficiently small. Remember also that the relation $a_{v_0,-,j_1}<a_{v_0,-,j_2}$ holds for an arbitrary pair $(j_1,j_2)$ with $j_1<j_2$, that the relation $a_{v_0,+,j_1}<a_{v_0,+,j_2}$ holds for an arbitrary pair $(j_1,j_2)$ with $j_1<j_2$, and that the relations $a_{v_0,-,j_{-}}<a_{v_0,+,j_{+}}$ and $a_{v_0}-{\epsilon}^{\prime}<a_{v_0,-,j_{-}}, a_{v_0,+,j_{+}}<a_{v_0}+{\epsilon}^{\prime}$ hold for an arbitrary pair $(j_{-},j_{+})$. By our definitions and assumptions on the sufficiently small numbers, we can regard the relations $0<{\epsilon}_{v_0,+,j},{\epsilon}_{v_0,+,j+1}<a_{v_0,+,j+1}-a_{v_0,+,j}<{\epsilon}^{\prime}<2{\epsilon}^{\prime}<\epsilon$ to be true and that the numbers ${\epsilon}_{v_0,+,j}$ and ${\epsilon}_{v_0,+,j+1}$ are sufficiently small. 

In addition, ${\pi}_{2,1,1}(v_0)$ is the minimum or the maximum of the function $c_G={\pi}_{2,1,1} \circ \tilde{c_G}$.

We can put the desired ellipsoid $D_{j^{\prime}}$ of the standard form centered at a point of the form $(\min \{{\pi}_{2,1,1}(v_0)-{\epsilon}_{v_0,+,j},{\pi}_{2,1,1}(v_0)-{\epsilon}_{v_0,+,j+1}\},a_{v_0,+,j,j+1})$ and this completes our proof of (\ref{thm:5.1}). This set $D_{j^{\prime}}$ is also regarded as one with $a_1$ in our definition "$\{x=(x_1,x_2) \in {\mathbb{R}}^2 \mid a_1{(x_1-x_{0,1})}^2+a_2{(x_2-x_{0,2})}^2 \leq r\}$" being sufficiently large and with $a_2$ being chosen as large as possible. We cannot choose $a_2$ sufficiently large. This is due to (\ref{thm:5.1.2}).

For this, see also FIGUREs \ref{fig:4} and \ref{fig:5}.
\begin{figure}
	\includegraphics[width=80mm,height=80mm]{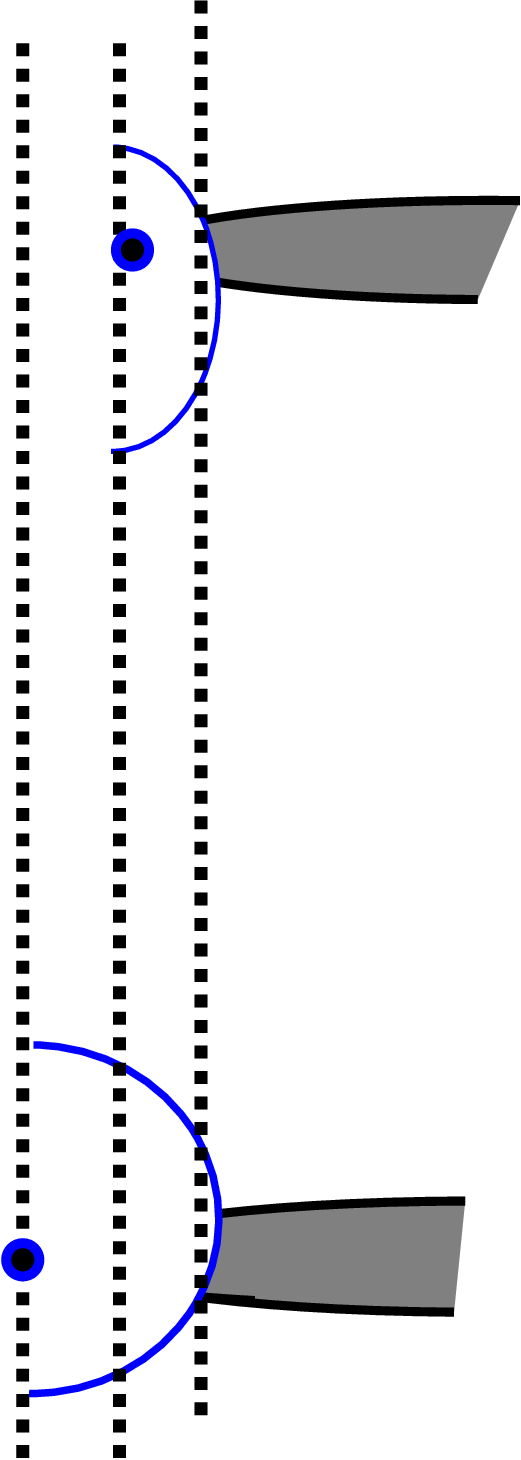}
	\caption{A pair of ellipsoids of the standard form containing the two points $({\pi}_{2,1,1}(v_0)-{\epsilon}_{v_0,+,j},a_{v_0,+,j})$ and $({\pi}_{2,1,1}(v_0)-{\epsilon}_{v_0,+,j+1},a_{v_0,+,j+1})$, colored in black with blue colored contours.}
	\label{fig:4}
\end{figure}
\begin{figure}
	\includegraphics[width=80mm,height=80mm]{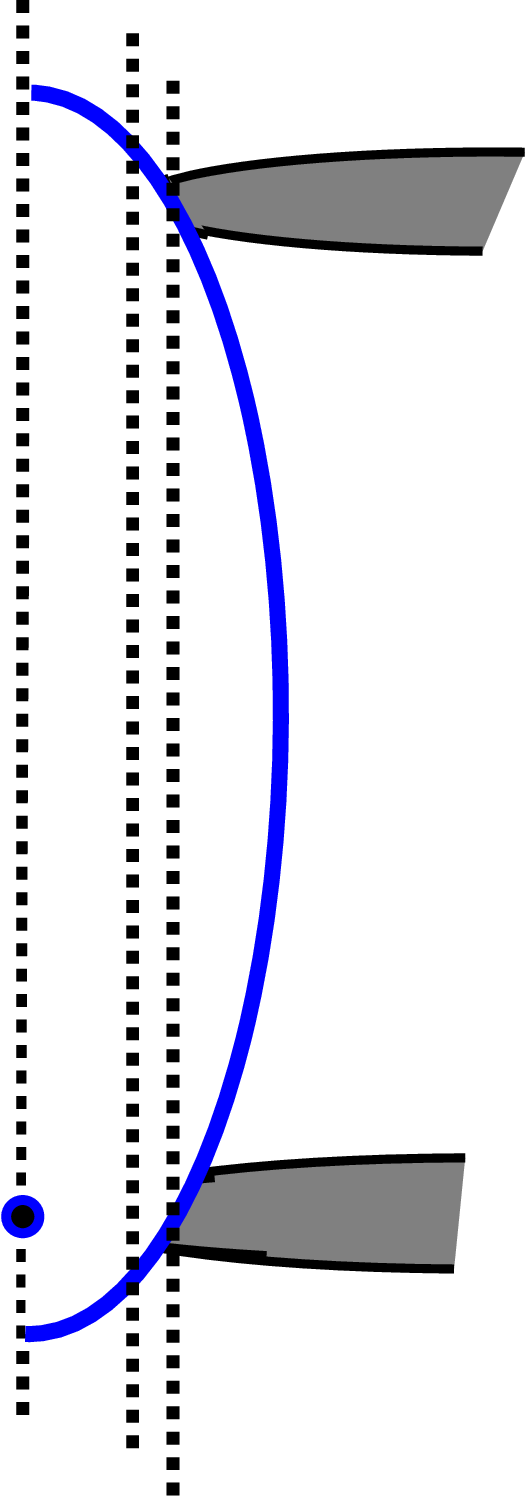}
	\caption{The desired ellipsoid $D_{j^{\prime}}$ of the standard form centered at a point of the form $(\min \{{\pi}_{2,1,1}(v_0)-{\epsilon}_{v_0,+,j},{\pi}_{2,1,1}(v_0)-{\epsilon}_{v_0,+,j+1}\}(={\pi}_{2,1,1}(v_0)-{\epsilon}_{v_0,+,j},a_{v_0,+,j,j+1})$. This respects FIGURE \ref{fig:4}.}
	\label{fig:5}
\end{figure}

We can have (\ref{thm:5.2}) by the symmetry.

(\ref{thm:5.3}) follows from the difference of the values of the components of the two points in (\ref{thm:5.1.2}).

This completes our proof.

\end{proof}
For a related topic, we present a result from \cite{kitazawa2} explicitly.
The {\it Reeb graph} of a smooth function on a manifold with no boundary is the graph which is also the space of all connected components of the preimages of single points for the function and the natural quotient space of the manifold and whose vertex set consists of all connected components containing some singular points of the function. Such objects are graphs in considerable situations from \cite{saeki1, saeki2} for example, and we can also check easily in explicit cases.  
We can also define the structure of V-digraphs for Reeb graphs canonically.

\begin{Thm}
\label{thm:6}
In Definition \ref{def:1}, let $A$ be a finite set of size $|A|$ and $S_j$ the zero set of a real polynomial $f_j$ of two variables, for each curve $S_j$, an element $a(j) \in A$ is assigned by the rule such that for two distinct curves $S_{j_1}$ and $S_{j_2}$ intersecting in $\overline{D_{\mathcal{S}}}$, distinct elements are assigned, that the map to $A$ is surjective, and that $\overline{D_{\mathcal{S}}}={\bigcap}_j \{x \in {\mathbb{R}}^2 \mid f_j(x) \geq 0\}$. Let $m_A$ be a function whose values are positive integers. Then $M:=\{(x,\{y_a\}_{a \in A}) \in {\mathbb{R}}^{({\Sigma}_{a \in A} m_A(a))+|A|+2} \mid a \in A, {\prod}_{j \in {m_A}^{-1}(a)} (f_j(x))-{\Sigma}_{j=1}^{m_A(a)+1} {y_{a,j}}^2=0\}$ with $y_a=(y_{a,1}, \cdots y_{a,m_A(a)})$ is the zero set of the real polynomial map defined by ${\prod}_{j \in {m_A}^{-1}(a)} f_j(x)$ {\rm (}$a \in A${\rm )} and non-singular. The restriction of ${\pi}_{m+|A|,2,1}$ there is a real algebraic map onto the closure $\overline{D_{\mathcal{S}}}$. The Reeb graph of the composition of the map with the projection ${\pi}_{2,1,i}$ is isomorphic to the Poincar\'e-Reeb V-digraph for ${\pi}_{2,1,i}$ as V-digraphs {\rm (}note that in the case some $S_j$ is a straight line we need additional and small attention{\rm :} it is a small problem and we do not need to understand theory related to Reeb graphs here{\rm )}.

\end{Thm}

For example, in Theorem \ref{thm:4}, we replace several pairs $(D_{{j_1}^{\prime}}, D_{{j_2}^{\prime}})$ of distinct ellipsoids into single ellipsoids and have Theorem \ref{thm:5}. This decreases the degrees of the polynomials in $M:=\{(x,\{y_a\}_{a \in A}) \in {\mathbb{R}}^{({\Sigma}_{a \in A} m_A(a))+2} \mid a \in A, {\prod}_{j \in {m_A}^{-1}(a)} (f_j(x))-{\Sigma}_{j=1}^{m_A(a)} {y_{a,j}}^2=0\}$ and makes the polynomials simpler in a sense.

Note also that from the simplest case $(\{S^1\},{D^2}^{\circ})$, we have the canonical projection ${\pi}_{m_A(a)+3,2,1} {\mid}_{S^{m_A(a)+2}}:S^{m_A(a)+2} \rightarrow {\mathbb{R}}^2$ with $|A|=1$.

Such observations are on explicit discoveries on reconstruction of nice and explicit real algebraic functions whose Reeb graphs are isomorphic to given graphs. Such studies are essentially founded by the author \cite{kitazawa1}. Originally, \cite{sharko}, followed by \cite{masumotosaeki}, is a pioneering study, studying cases of reconstruction of differentiable (smooth) functions on closed surfaces.





\section{Conflict of interest and Data availability.}
\noindent {\bf Conflict of interest.}
The author works at Institute of Mathematics for Industry (https://www.jgmi.kyushu-u.ac.jp/en/about/young-mentors/) and the present work is closely related to our study. We thank them for supports and encouragement. The author is also a researcher at Osaka Central
Advanced Mathematical Institute (OCAMI researcher), supported by MEXT Promotion of Distinctive Joint Research Center Program JPMXP0723833165. He is not employed there. We also thank them for the hospitality. The author would also like to thank the conference "Singularity theory of differentiable maps and its applications" (https://www.fit.ac.jp/$\sim$fukunaga/conf/sing202412.html) for letting the author to present related result \cite{kitazawa1, kitazawa2, kitazawa4} of the author. Presented comments have helped the author to study further including the present study. The conference is supported by the Research Institute for Mathematical Sciences, an International Joint Usage/Research Center located in Kyoto University. \\
\ \\
{\bf Data availability.} \\
Data essentially related to our present study are all in the present file.

\end{document}